\RequirePackage{ifpdf}
\ifpdf 
\documentclass[pdftex]{sigma}
\else
\documentclass{sigma}
\fi

\newtheorem{thm}{Theorem}[section]
\newtheorem{cor}[thm]{Corollary}
\newtheorem{lem}[thm]{Lemma}
\newtheorem{propo}[thm]{Proposition}

\numberwithin{equation}{section}

\begin{document}
\allowdisplaybreaks

\newcommand{\arXivNumber}{1704.01597}

\renewcommand{\thefootnote}{}

\renewcommand{\PaperNumber}{024}

\FirstPageHeading

\ShortArticleName{Fourier Series of Gegenbauer--Sobolev Polynomials}

\ArticleName{Fourier Series of Gegenbauer--Sobolev Polynomials\footnote{This paper is a~contribution to the Special Issue on Orthogonal Polynomials, Special Functions and Applications (OPSFA14). The full collection is available at \href{https://www.emis.de/journals/SIGMA/OPSFA2017.html}{https://www.emis.de/journals/SIGMA/OPSFA2017.html}}}

\Author{\'Oscar CIAURRI and Judit M\'INGUEZ}

\AuthorNameForHeading{\'O.~Ciaurri and J.~M\'{\i}nguez}

\Address{Departamento de Matem\'aticas y Computaci\'on, Universidad de La Rioja, 26006 Logro\~no, Spain}
\Email{\href{mailto:oscar.ciaurri@unirioja.es}{oscar.ciaurri@unirioja.es}, \href{mailto:judit.minguez@unirioja.es}{judit.minguez@unirioja.es}}

\ArticleDates{Received January 19, 2018, in final form March 13, 2018; Published online March 17, 2018}

\Abstract{We study the partial sum operator for a Sobolev-type inner product related to the classical Gegenbauer polynomials. A complete characterization of the partial sum operator in an appropriate Sobolev space is given. Moreover, we analyze the convergence of the partial sum operators.}


\Classification{42A20; 33C47}

\renewcommand{\thefootnote}{\arabic{footnote}}
\setcounter{footnote}{0}

\section{Introduction}

Let the Sobolev-type inner product be
\begin{gather}
\langle f,g \rangle_S=\int_{-1}^{1}f(x)g(x){\rm d}\mu_{\alpha}(x)+M(f(1)g(1)+f(-1)g(-1))\nonumber\\
\hphantom{\langle f,g \rangle_S=}{} +N(f'(1)g'(1)+f'(-1)g'(-1)),\label{inner-product}
\end{gather}
where $M\ge 0$, $N\ge 0$, and
\begin{gather*}
{\rm d}\mu_{\alpha}(x)=\frac{\Gamma(2\alpha+2)}{2^{2\alpha+1}\Gamma^2(\alpha+1)}\big(1-x^2\big)^{\alpha}{\rm d}x,\qquad \alpha>-1/2,
\end{gather*}
is the probability measure corresponding to the Gegenbauer polynomials.

Let $\{Q_n^{\alpha}(x)\}_{n\ge 0}$ be the sequence of normalized Gegenbauer--Sobolev orthonormal polynomials with respect to the inner product~\eqref{inner-product}. For each appropriate function~$f$, we define its sequence of Fourier--Gegenbauer--Sobolev coefficients by
\begin{gather*}
\widehat{f}(k)=\langle f, Q_k^{\alpha}\rangle_S, \qquad k=0,1,\dots,
\end{gather*}
and the $n$-th partial sum operator as
\begin{gather*}
G_nf(x)=\sum_{k=0}^n \widehat{f}(k) Q_k^{\alpha}(x), \qquad n=0,1,\dots.
\end{gather*}

Given $1< p<\infty$, we say that $f\in L^p({\rm d}\mu_{\alpha})$ if $f$ is a measurable function in~$[-1,1]$ and
\begin{gather*}
\|f\|_{L^p({\rm d}\mu_{\alpha})}=\left(\int_{-1}^1 |f(x)|^p {\rm d}\mu_{\alpha}(x)\right)^{1/p}<\infty.
\end{gather*}
Let us define the measure $\mu_{\alpha,M} := \mu_\alpha+M(\delta_1+\delta_{-1})$. We consider the space $W_p^{\alpha}$, with $1< p<\infty$, as the set of equivalence classes, with respect to the (semi)norm in $L^p(\mu_{\alpha,M})$, of measurable functions defined on $[-1,1]$ such that there exists an element in the class~$f$ for which~$f'(1)$ and~$f'(-1)$ are defined, and
\begin{gather*}
\|f\|_{W_p^{\alpha}}^p:=\|f\|_{L^p({\rm d}\mu_{\alpha})}^p+M\big(|f(1)|^p+|f(-1)|^p\big) +N\big(|f'(1)|^p+|f'(-1)|^p\big)<\infty.
\end{gather*}

The main target of this paper is the study of the uniform boundedness of the operators $G_n$. In fact, we will prove the following characterization.
\begin{thm}\label{main}
	Let $\alpha>-1/2$, $1<p<\infty$, and $f\in W_p^{\alpha}$. There exists a constant $C$, independent of $n$ and $f$, such that
	\begin{gather*}
	\|G_nf\|_{W_p^{\alpha}}\le C\|f\|_{W_p^{\alpha}}
	\end{gather*}
	if and only if
	\begin{gather}\label{acotacionp}
	\frac{4(\alpha+1)}{2\alpha+3}<p<\frac{4(\alpha+1)}{2\alpha+1}.
	\end{gather}
\end{thm}

The uniform boundedness of the partial sum operators for Gegenbauer polynomials in $L^p({\rm d}\mu_{\alpha})\!$ was given by Pollard~\cite{PollardII} who extended it to the Jacobi setting in \cite{PollardIII}. A general result including weights for Jacobi expansions can be seen in~\cite{muckenhoupt}. In~\cite{GPRV1}, by applying the boundedness with weights of the Hilbert transform, the authors did a complete study of the boundedness of the partial sum operators related to generalized Jacobi weights. The same authors studied the generalized Jacobi weights with mass points on the interval $[-1,1]$ (see~\cite{GPRV2}). The uniform boundedness with weights of the partial sum operator for the generalized Jacobi polynomials has been used to proved, using an idea dating back to J.~Marcinkiewicz, some results related to interpolating polynomials (see \cite{Xu-93,Xu-94} and the references in \cite{Nevai}).

It would be natural to consider our problem for the Jacobi weight instead of the Gegenbauer one. This extension requires some results about the corresponding Jacobi--Sobolev polyno\-mials that are unavailable in the literature at this moment. We hope to develop these tools in a~forthcoming paper to obtain a~complete characterization in that case as well.

As far as we know, a complete characterization of the uniform boundedness of the partial sums in the Sobolev setting is completely new. In~\cite{Marce-Xu}, the authors observed that the main obstacle to analyze this problem is the lack of Christoffel--Darboux formula for Sobolev orthogonal polynomials. As a consequence of this fact, except for certain particular cases, the convergence of Fourier expansions in Sobolev orthogonal polynomials has not been resolved. For example, the particular case of the Fourier series associated to the Jacobi--Sobolev polynomials defined by the inner product
\begin{gather*}
\int_{-1}^{1} f(x)g(x)(1-x)^\alpha(1+x)^\beta {\rm d}x+ \int_{-1}^{1} f'(x)g'(x)(1-x)^{\alpha+1}(1+x)^{\beta+1} {\rm d}x
\end{gather*}
was treated in \cite{MQU} but, unfortunately, the given results are not completely satisfactory.

Our proof of Theorem \ref{main} relies on some results about multipliers and transplantation ope\-ra\-tors for Jacobi expansions proved by Muckenhoupt and other authors in the eighties of the last century (see \cite{muckenhoupt1} and the references therein).

From a standard argument, the uniform boundedness of the operator $G_n$ will imply the convergence for functions in the class $W_p^\alpha$ if the polynomials form a dense class. However, the reverse implication is not true because the space $W_p^{\alpha}$ is not complete. The density of the polynomials is contained in the next result.

\begin{thm}\label{density}
The set of polynomials is dense in the space $W_p^{\alpha}$. That is, given $f\in W_p^{\alpha}$, for all $\varepsilon>0$ there exists a polynomial $q_n$ of degree $n$ such that
\begin{gather*}
\|f-q_n\|_{W_p^{\alpha}}<\varepsilon.
\end{gather*}
\end{thm}

Now, from Theorems \ref{main} and \ref{density}, we deduce the convergence for functions in the class $W_p^{\alpha}$ of the partial sums $G_n$.

\begin{cor}Let $f\in W_p^{\alpha}$ with $\alpha>-1/2$ and $1<p<\infty$. If
	\begin{equation*}
	\frac{4(\alpha+1)}{2\alpha+3}<p<\frac{4(\alpha+1)}{2\alpha+1},
	\end{equation*}
then
\begin{gather*}
 \lim_{n\to\infty}\|G_nf-f\|_{W_p^{\alpha}}=0.
\end{gather*}
\end{cor}

In Section~\ref{section2} we present the necessary definitions and results concerning to the Gegenbauer and Gegenbauer--Sobolev polynomials. Section~\ref{section3} and Section~\ref{section4} are devoted to prove Theorem~\ref{main} and Theorem~\ref{density}, respectively.

\section{Definitions and auxiliary results}\label{section2}

Let $\{R_n^{\alpha}\}_{n\ge 0}$ be the sequence of Gegenbauer polynomials given by the Rodrigues formula
\begin{gather*}
R_n^{\alpha}(x)=\frac{(-1)^n}{\Gamma(\alpha+1)}{2^n\Gamma(n+\alpha+1)}\big(1-x^2\big)^{-\alpha} \frac{{\rm d}^n}{{\rm d}x^n}\big(\big(1-x^2\big)^{n+\alpha}\big).
\end{gather*}
If we call $\{B_n^{\alpha}\}_{n_\ge 0}$ the sequence of orthogonal polynomials with respect to~\eqref{inner-product}, the following relation between $R_n^{\alpha}$ and $B_n^{\alpha}$ was proved in \cite{bavinck}
\begin{gather}
B_n^{\alpha}(x)=\frac{a_n(n+2\alpha+1)_4(-n)_4}{2^6(\alpha+2)(\alpha+3)(\alpha+1)_4}\big(1-x^2\big)^2R_{n-4}^{\alpha+4}(x)\nonumber\\
\hphantom{B_n^{\alpha}(x)=}{} +\frac{b_n(n+2\alpha+1)_2(-n)_2}{2^2(\alpha+1)_2}\big(1-x^2\big)R_{n-2}^{\alpha+2}(x)+c_nR_n^{\alpha}(x),\label{rel-ortogonal}
\end{gather}
where $(a)_n$ is the shifted factorial (or Pochhammer symbol), defined by $(a)_n=\frac{\Gamma(a+n)}{\Gamma(a)}$, and
\begin{gather*}
a_n=MN\frac{4(2\alpha+3)_n(2\alpha+3)_{n-2}}{(\alpha+1)(\alpha+2)n!(n-2)!}+N\frac{2(2\alpha+3)_{n-1}}{(\alpha+1)(n-1)!},\\
b_n =-\frac{N}{2}\frac{(2\alpha+3)_{n-1}(n-2)(n+2\alpha+3)}{(\alpha+1)(\alpha+3)(n-1)!}-2M\frac{(2\alpha+3)_{n-2}}{n!},\\
c_n=1-\frac{N}{2}\frac{(2\alpha+3)_{n+1}}{(\alpha+1)(\alpha+2)(\alpha+3)(n-3)!}.
\end{gather*}
Here and elsewhere we use the convention that $R_n^{\gamma}\equiv 0$ if $n<0$.

In \cite{foulquie} it was proved the identity
\begin{gather*}
\|B_n^\alpha\|_{W_2^\alpha}^2=2Mc_n^2+2N\left(\frac{n(n+2\alpha+1)}{2(\alpha+1)}+
\frac{M(2\alpha+3)_n}{(\alpha+1)(\alpha+2)(n-1)!}\right)^2\\
\hphantom{\|B_n^\alpha\|_{W_2^\alpha}^2=}{}
+\frac{\Gamma(2\alpha+2) n!}{(2n+2\alpha+1) \Gamma(n+2\alpha+1)}\left(\frac{n(n-1)(n-2)(n-3)(n+2\alpha+1)_4 a_n^2}{16(\alpha+2)^2(\alpha+3)^2}\right.\\
\hphantom{\|B_n^\alpha\|_{W_2^\alpha}^2=}{} \times n(n-1)(n+2\alpha+1)_2 b_n^2-
\frac{n(n-1)(n-2)(n-3)(n+2\alpha+1)_2 a_n b_n}{2(\alpha+2)_2}\\\left.
\hphantom{\|B_n^\alpha\|_{W_2^\alpha}^2=}{} +c_n^2
-2n(n-1)b_nc_n+\frac{n(n-1)(n-2)(n-3)a_nc_n}{2(\alpha+2)(\alpha+3)}\right).
\end{gather*}
Then $Q_n^{\alpha}(x)=\lambda_{n,\alpha} B_n^{\alpha}(x)$, where $\lambda_{n,\alpha}^{-2}=\|B_n^\alpha\|_{W_2^\alpha}^2$. Now, denoting by $\{P_n^{\alpha}\}_{n\ge 0}$ the sequence of orthonormal Gegenbauer polynomials, given by $P_n^\alpha=\beta_{n,\alpha} R_n^\alpha$ with
\begin{gather*}
\beta_{n,\alpha}^{-2}=\|R_n^{\alpha}\|_{L^2({\rm d}\mu_{\alpha})}^2=\frac{\Gamma(2\alpha+2)n!}{(2n+2\alpha+1)\Gamma(n+2\alpha+1)},
\end{gather*}
from \eqref{rel-ortogonal} we can write
\begin{gather}\label{rel-ortonormal}
Q_n^{\alpha}(x)=A_{n,4}\big(1-x^2\big)^2P_{n-4}^{\alpha+4}(x)+A_{n,2}\big(1-x^2\big)P_{n-2}^{\alpha+2}(x)+A_{n,0}P_n^{\alpha}(x),
\end{gather}
where
\begin{gather*}
A_{n,4}=\frac{a_n(n+2\alpha+1)_4(-n)_4}{2^6(\alpha+2)(\alpha+3)(\alpha+1)_4}\frac{\lambda_{n,\alpha}}{\beta_{n-4,\alpha+4}}, \qquad
A_{n,2}=\frac{b_n(n+2\alpha+1)_2(-n)_2}{2^2(\alpha+1)_2}\frac{\lambda_{n,\alpha}}{\beta_{n-2,\alpha+2}},
\end{gather*}
and $A_{n,0}=c_n\frac{\lambda_{n,\alpha}}{\beta_{n,\alpha}}$.

We consider the notations
\begin{gather*}
g(n,J)=\sum_{j=0}^{J-1}\frac{d_j}{(n+1)^{j}}+O\left(\frac{1}{(n+1)^J}\right), \qquad J\in \mathbb{N},\\
h(n,\alpha)=\sum_{j=0}^{\lfloor 2\alpha+2 \rfloor}\frac{D_j}{(n+1)^{j}}+O\left(\frac{1}{(n+1)^{2\alpha+2}}\right), \qquad \alpha>-1/2,
\end{gather*}
for some constants $d_j$ and $D_j$ that will be different in each occurrence of the $g(n,J)$ and $h(n,\alpha)$, respectively. With the previous notation, by using that
\begin{gather}\label{ec:gamma-ratio}
\frac{\Gamma(n+a)}{\Gamma(n+b)}=n^{a-b}g(n,J),
\end{gather}
for any $J\in \mathbb{N}$, we deduce in an easy way that
\begin{gather}\label{lambda_n}
\lambda_{n,\alpha}= \begin{cases}
(n+1)^{-\alpha-\frac{11}{2}}g(n,J), & M=0, \ N>0,\\
(n+1)^{-3\alpha-\frac{15}{2}}h(n,\alpha), & M>0, \ N>0,\\
(n+1)^{-\alpha-\frac{3}{2}}h(n,\alpha), & M>0, \ N=0,
\end{cases}
\end{gather}
for any $J\in \mathbb{N}$.

\begin{lem}\label{A_nB_nC_n} Let $\alpha>-1/2$. Then the constants $A_{n,4}$, $A_{n,2}$, and $A_{n,0}$ in~\eqref{rel-ortonormal} satisfy the following:
	\begin{itemize}\itemsep=0pt
		\item [$i)$] If $M=0$ and $N>0$,
		\begin{gather*}
		A_{n,4}=g(n,J),\qquad A_{n,2}= g(n,J),\qquad A_{n,0}= g(n,J),
		\end{gather*}
for any $J\in \mathbb{N}$.
		\item [$ii)$] If $M>0$ and $N>0$,
		\begin{gather*}
		A_{n,4}= h(n,\alpha),\qquad A_{n,2}=\frac{h(n,\alpha)}{(n+1)^{2\alpha+2}},\qquad A_{n,0}=\frac{h(n,\alpha)}{(n+1)^{2\alpha+2}}.
		\end{gather*}
		\item [$iii)$] If $M>0$ and $N=0$,
		\begin{gather*}
		A_{n,4}=0,\qquad A_{n,2}=h(n,\alpha),\qquad A_{n,0}=\frac{h(n,\alpha)}{(n+1)^{2\alpha+2}}.
		\end{gather*}
	\end{itemize}
\end{lem}

\begin{proof} From \eqref{ec:gamma-ratio} we have
\begin{gather*}
a_{n}= \begin{cases}
(n+1)^{2\alpha+2}g(n,J), & M=0, \ N>0,\\
(n+1)^{4\alpha+4}h(n,\alpha), & M>0, \ N>0,\\
0, & M>0, \ N=0,
\end{cases} \\
b_{n}= \begin{cases}
(n+1)^{2\alpha+4}g(n,J), & M=0, \ N>0,\\
(n+1)^{2\alpha+4}g(n,J), & M>0, \ N>0,\\
(n+1)^{2\alpha}g(n,J), & M>0, \ N=0,
\end{cases} \\
c_{n}= \begin{cases}
(n+1)^{2\alpha+6}g(n,J), & M=0, \ N>0,\\
(n+1)^{2\alpha+6}g(n,J), & M>0, \ N>0,\\
1, & M>0, \ N=0,
\end{cases}
\end{gather*}
and $\beta_{n,\alpha}=(n+1)^{\alpha+1/2}g(n,J)$, for any $j\in \mathbb{N}$. Then, using that \eqref{lambda_n}, the result follows.
\end{proof}

The following results, that we will use in the proof of Theorem~\ref{main}, can be found in~\cite{foulquie}. The notation appearing in Lemma~\ref{extremos}, $f(n)\approx g(n)$, indicates the existence of positive constants~$C$ and $D$ such that $Cf(n)\le g(n)\le Df(n)$ for $n$ large enough.

\begin{lem}\label{acotacionQ}
	Let $\{Q_n^{\alpha}\}_n$ be the sequence of orthonormal polynomials with respect to the inner product \eqref{inner-product}, then
	\begin{gather*}
	\max_{-1\le x\le 1}\big(1-x^2\big)^{\frac{\alpha}{2}+\frac{1}{4}}|Q_n^{\alpha}(x)|\le C.
	\end{gather*}
\end{lem}

\begin{lem}\label{extremos} Let $\{Q_n^{\alpha}\}_n$ be the sequence of orthonormal polynomials with respect to the inner product~\eqref{inner-product}, then
\begin{gather*}
|Q_n^{\alpha}(1)|=|Q_n^{\alpha}(-1)|\approx \begin{cases}
(n+1)^{-\alpha-3/2}, & M>0, \ N\ge 0,\\
(n+1)^{\alpha+1/2}, & M=0, \ N>0,
\end{cases}\\
|(Q_n^{\alpha})'(1)|=|(Q_n^{\alpha})'(-1)|\approx \begin{cases}
(n+1)^{-\alpha-7/2}, & M\ge 0, \ N> 0,\\
(n+1)^{\alpha+5/2}, & M>0, \ N=0.
\end{cases}
\end{gather*}
\end{lem}

Let $S_n^{\gamma}f$ be the $n$-th partial sum of Fourier expansion in terms of orthonormal Gegenbauer polynomials,
\begin{gather*}
S_n^{\gamma}f(x)=\sum_{k=0}^n d_k^{\gamma}(f)P_k^{\gamma}(x),\qquad d_k^{\gamma}(f)=\int_{-1}^1 f(x)P_k^{\gamma}(x) {\rm d}\mu_{\gamma}(x).
\end{gather*}
From the main result in \cite{muckenhoupt1} we can deduce the following result
\begin{lem}\label{acotacion}
	Let $\gamma>-1$ and $1<p<\infty$. There exists a constant $C$, independent of $n$ and $f$, such that
	\begin{gather*}
	\big\|\big(1-(\cdot)^2\big)^a S^{\gamma}_nf\big\|_{L^p({\rm d}\mu_{\gamma})}\le C\big\|\big(1-(\cdot)^2\big)^a f\big\|_{L^p ({\rm d}\mu_{\gamma})}
	\end{gather*}
	if and only if
	\begin{gather*}
	\left|a+(\gamma+1)\bigg(\frac{1}{p}-\frac{1}{2}\bigg)\right|<\frac{1}{4}.
	\end{gather*}
\end{lem}

Let $d$ be an integer number. We define the transplantation operator
\begin{gather*}
T_{d}^{\beta,\gamma}f(x)=\sum_{k=0}^{\infty}d_k^{\gamma}(f)P_{k+d}^{\beta}(x).
\end{gather*}
The operator $T_{d}^{\beta,\gamma}$ is well defined, for example, for functions $f$ having a finite expansion in terms of the Gegenbauer polynomials $P_n^\gamma$. The following result plays a crucial role in our work. It is essentially a special case of a general weighted transplantation theorem due to Muckenhoupt, see \cite[Theorem 1.6]{muckenhoupt}.

\begin{lem}\label{mukchenhoupt}
	Let $\gamma>-1$, $\beta>-1$, and $1<p<\infty$. If $2(b+1)>-p(\beta+1/2)$ then	
	\begin{gather*}
	\left(\int_{-1}^{1}|T_{d}^{\beta,\gamma}f(x)|^p\big(1-x^2\big)^{\frac{p}{2}(\beta+1/2)+b} {\rm d}x\right)^{1/p} \le C\left(\int_{-1}^{1}|f(x)|^p\big(1-x^2\big)^{\frac{p}{2}(\gamma+1/2)+b} {\rm d}x\right)^{1/p}.
	\end{gather*}
\end{lem}

The last tool that we will need for the proof of Theorem \ref{main} is related to the boundedness of a specific multiplier for Gegenbauer expansions. We define the operator
\begin{gather*}
R^\gamma f(x)=\sum_{k=0}^\infty \frac{d_k^\gamma(f)}{k+1}P_{k}^\gamma(x).
\end{gather*}
\begin{lem}\label{lem:muck}
Let $\gamma>-1$ and $1<p<\infty$. If $|2b+1|<p$ and
\begin{gather*}
\left|\frac{2(b+1)}{p}-\frac{1}{2}\right|<\min\{\gamma+1,1/2\},
\end{gather*}
then	
	\begin{gather*}
	\left(\int_{-1}^{1}|R^\gamma f(x)|^p\big(1-x^2\big)^{\frac{p}{2}(\gamma+1/2)+b} {\rm d}x\right)^{1/p} \le C\left(\int_{-1}^{1}|f(x)|^p\big(1-x^2\big)^{\frac{p}{2}(\gamma+1/2)+b} {\rm d}x\right)^{1/p}.
	\end{gather*}
\end{lem}

This lemma is a particular case of \cite[Theorem~1.10]{muckenhoupt} because the multiplier $1/(k+1)$ belongs to the class $M(1,1)$ there defined.

\section{Proof of Theorem \ref{main}}\label{section3}
Taking the kernel
\begin{gather*}
L_n(x,y)=\sum_{k=0}^n Q_k^{\alpha}(x)Q_k^{\alpha}(y),
\end{gather*}
it is easy to see that
\begin{gather*}
G_nf(x)=\langle L_n(x,y),f\rangle_S.
\end{gather*}
Recall that
\begin{gather*}
\|G_nf\|_{W_p^{\alpha}}^p=\|G_nf\|_{L^p({\rm d}\mu_{\alpha})}^p+M\big(|G_nf(1)|^p+|G_nf(-1)|^p\big)\\
\hphantom{\|G_nf\|_{W_p^{\alpha}}^p=}{} +N\big(|(G_nf)'(1)|^p+(G_nf)'(-1)|^p\big).
\end{gather*}

The necessity of the condition \eqref{acotacionp} is a consequence of \cite[Theorem 1]{Fejzullahu} and its sufficiency will be obtained from two following propositions.
\begin{propo}\label{prop1} Let $\alpha>-1/2$ and $1<p<\infty$. If \eqref{acotacionp} holds, then
	\begin{gather}\label{des-1}
	\|G_nf(x)\|_{L^p({\rm d}\mu_{\alpha})}\le C\|f\|_{{W_p^{\alpha}}},
	\end{gather}
where $C$ is a constant independent of $n$ and $f$.
\end{propo}

\begin{propo}\label{prop2} Let $\alpha>-1/2$ and $1<p<\infty$. If \eqref{acotacionp} holds, then
	\begin{equation*}
M\big(|G_nf(1)|^p+|G_nf(-1)|^p\big)+N\big(|(G_nf)'(1)|^p+|(G_nf)'(-1)|^p\big)\le C\|f\|_{W_p^{\alpha}},
	\end{equation*}
where $C$ is a constant independent of $n$ and $f$.
\end{propo}

\begin{proof}[Proof of Proposition \ref{prop1}] From Minkowski's inequality, we know that
	\begin{gather*}
	\|G_nf(x)\|_{L^p({\rm d}\mu_{\alpha})}\le
	\left(\int_{-1}^{1}\left|\int_{-1}^{1}f(y)L_n(x,y) {\rm d}{\mu_{\alpha}}(y)\right|^p {\rm d}\mu_{\alpha}(x)\right)^{1/p}\\
\hphantom{\|G_nf(x)\|_{L^p({\rm d}\mu_{\alpha})}\le}{}	+\left(\int_{-1}^{1}|M(f(1)L_n(x,1)+f(-1)L_n(x,-1)|^p {\rm d}\mu_{\alpha}(x)\right)^{1/p}\\
\hphantom{\|G_nf(x)\|_{L^p({\rm d}\mu_{\alpha})}\le}{}	+\left(\int_{-1}^{1}\left|N(f'(1)\frac{\partial L_n}{\partial y}(x,1)+f'(-1)\frac{\partial L_n}{\partial y}(x,-1))\right|^p {\rm d}\mu_{\alpha}(x)\right)^{1/p}.
	\end{gather*}	
First, it will be proved that
\begin{gather}\label{des-des}
\left(\int_{-1}^{1}\left|\int_{-1}^{1}f(x)L_n(x,y) {\rm d}\mu_{\alpha}(y)\right|^p {\rm d}\mu_{\alpha}(x)\right)^{1/p}\le C\|f\|_{L^p({\rm d}\mu_{\alpha})}.
	\end{gather}
Using \eqref{rel-ortonormal}, we have
	\begin{gather*}
	\int_{-1}^{1}f(y)L_n(x,y) {\rm d}\mu_{\alpha}(y)=\sum_{j,m\in \{4,2,0\}}M_n^{j,m}f(x),
	\end{gather*}
where
\begin{gather*}
M_n^{j,m}f(x)=\int_{-1}^{1} f(y) K_n^{j,m}(x,y) {\rm d}\mu_\alpha(y),\\
K_n^{j,m}(x,y)=\big(1-x^2\big)^{j/2}\big(1-y^2\big)^{m/2}\sum_{k=0}^{n}A_{k,j}A_{k,m}P_{k-j}^{\alpha+j}(x)P_{k-m}^{\alpha+m}(y).
\end{gather*}
By using a standard duality argument, to deduce \eqref{des-des} it is enough to prove
\begin{gather*}
\|M_n^{j,m}f\|_{L^p({\rm d}\mu_{\alpha})}\le C\|f\|_{L^p({\rm d}\mu_{\alpha})}
\end{gather*}
for $m\le j$.

By Lemma \ref{A_nB_nC_n}, each operator $M_n^{j,m}$ can be decomposed as
\begin{gather*}
M_n^{j,m}f(x)=S_0 M_n^{j,m,0}f(x)+S_1 M_n^{j,m,1}f(x)+S_2 M_n^{j,m,2}f(x),
\end{gather*}
for some nonnegative constants $S_0$, $S_1$ and $S_2$, with
\begin{gather*}
\begin{split} & M_n^{j,m,s}f(x)=\int_{-1}^{1} f(y) K_n^{j,m,s}(x,y) {\rm d}\mu_\alpha(y), \qquad s=0,1,2,\\
& K_n^{j,m,s}(x,y)=\big(1-x^2\big)^{j/2}\big(1-y^2\big)^{m/2}\sum_{k=0}^{n}(k+1)^{-s}P_{k-j}^{\alpha+j}(x)P_{k-m}^{\alpha+m}(y), \qquad s=0,1,
\end{split}
\end{gather*}
and
\begin{gather*}
\big|K_n^{j,m,2}(x,y)\big|\le C\big(1-x^2\big)^{j/2}\big(1-y^2\big)^{m/2}\sum_{k=0}^{n}k^{-\theta}\big|P_{k-j}^{\alpha+j}(x)\big|\big|P_{k-m}^{\alpha+m}(y)\big|,
\end{gather*}
where $\theta>1$.

From the well-known estimate (it follows from \cite[Theorem~7.32.2, p.~169]{Szego})
\begin{gather*}
|P_n^\alpha(x)|\le C\big(1-x^2\big)^{-(\alpha/2+1/4)},\qquad x\in [-1,1],
\end{gather*}
with $C$ a constant independent of $n$, we deduce
\begin{gather*}
\big\|\big(1-(\cdot)^2\big)^{j/2}P_{n-j}^{\alpha+j}\big\|_{L^p({\rm d}\mu_\alpha)}\le C
\end{gather*}
for $p<4(\alpha+1)/(2\alpha+1)$. In this way, applying H\"{o}lder inequality,
\begin{gather*}
\big\|M_n^{j,m,2}f\big\|_{L^p({\rm d}\mu_{\alpha})}\le C\|f\|_{L^p({\rm d}\mu_{\alpha})},
\end{gather*}
for each $p$ verifying \eqref{acotacionp}.

It is easy to check that
\begin{gather*}
M_{n}^{j,j,0}f(x)=K_\alpha\big(1-x^2\big)^{j/2}S_{n-j}^{\alpha+j}g(x),
\end{gather*}
for a constant $K_\alpha$, with $j=4,2,0$ and $g(x)=\big(1-x^2\big)^{-j/2}f(x)$. Then, if $p$ satisfies \eqref{acotacionp}, from Lemma \ref{acotacion}, with $a=j(1/2-1/p)$ and $\gamma=\alpha+j$, we deduce
\begin{gather*}
\big\|M_n^{j,j,0}f\big\|_{L^p({\rm d}\mu_\alpha)} =K_\alpha\big\|\big(1-(\cdot)^2\big)^{j(1/2-1/p)}S_{n-j}^{\alpha+j}g\big\|_{L^p({\rm d}\mu_{\alpha+j})}\\
 \hphantom{\big\|M_n^{j,j,0}f\big\|_{L^p({\rm d}\mu_\alpha)}}{} \le C\big\|\big(1-(\cdot)^2\big)^{j(1/2-1/p)}g\big\|_{L^p({\rm d}\mu_{\alpha+j})}\le C \|f\|_{L^p({\rm d}\mu_\alpha)}.
\end{gather*}

Now, for $m<j$, we can check that
\begin{gather*}
M_n^{j,m,0}f(x)=C_\alpha \big(1-x^2\big)^{j/2}T_{m-j}^{\alpha+j,\alpha+m}\big(S_{n-m}^{\alpha+m}h\big)(x),
\end{gather*}
for a constant $C_\alpha$, whith $h(x)=\big(1-x^2\big)^{-m/2}f(x)$. So, using Lemma~\ref{mukchenhoupt} with $\beta=\alpha+j$, $\gamma=\alpha+m$, and $b=\alpha-p(\alpha+1/2)/2$, we have
\begin{gather*}
\big\|M_n^{j,m,0}f\big\|_{L^p({\rm d}\mu_\alpha)} =C_\alpha\big\|\big(1-(\cdot)^2\big)^{j/2}T_{m-j}^{\alpha+j,\alpha+m}\big(S_{n-m}^{\alpha+m}h\big)\big\|_{L^p({\rm d}\mu_{\alpha})}\\
\hphantom{\big\|M_n^{j,m,0}f\big\|_{L^p({\rm d}\mu_\alpha)}}{} \le C\big\|\big(1-(\cdot)^2\big)^{m/2}S_{n-m}^{\alpha+m}h\big\|_{L^p({\rm d}\mu_{\alpha})}\le C \|f\|_{L^p({\rm d}\mu_\alpha)},
\end{gather*}
where in the last step we have used Lemma~\ref{acotacion} as we have done for~$M_n^{j,j,0}$.

To analyze the operators $M_n^{j,m,1}$ we observe the identities
\begin{gather*}
M_{n}^{j,j,1}f(x)=K_\alpha \big(1-x^2\big)^{j/2}R^{\alpha+j}\big(S_{n-j}^{\alpha+j}g\big)(x), \qquad j=4,2,0,\\
M_n^{j,m,1}f(x)=C_\alpha \big(1-x^2\big)^{j/2}T_{m-j}^{\alpha+j,\alpha+m}\big(R^{\alpha+m}\big(S_{n-m}^{\alpha+m}h\big)\big)(x),\qquad m<j,
\end{gather*}
with $g(x)=\big(1-x^2\big)^{-j/2}f(x)$ and $h(x)=\big(1-x^2\big)^{-m/2}f(x)$. Then the boundedness of these operators follows as in the previous cases but using moreover the estimate
\begin{gather*}
\big\|\big(1-(\cdot)^2\big)^{j/2}R^{\alpha+j}f\big\|_{L^p({\rm d}\mu_\alpha)}\le C \big\|\big(1-(\cdot)^2\big)^{j/2}f\big\|_{L^p({\rm d}\mu_\alpha)},
\end{gather*}
which can be deduced from Lemma~\ref{lem:muck} taking $b=\alpha-p(\alpha+1/2)/2$ and $\gamma=\alpha+j$ under the assumption \eqref{acotacionp}. In this way the proof of~\eqref{des-des} is completed.

To finish the proof of \eqref{des-1}, we are going to prove the estimates
\begin{gather}
	\int_{-1}^1|M(f(1)L_n(x,1)+f(-1)L_n(x,-1)|^p\big(1-x^2\big)^{\alpha} {\rm d}x
\le C M^p \big(|f(1)|^p+|f(-1)|^p\big),\label{des-2}\\
	\int_{-1}^{1}\left|N\left(f'(1)\frac{\partial L_n}{\partial y}(x,1)+f'(-1)\frac{\partial L_n}{\partial y}(x,-1)\right)\right|^p \big(1-x^2\big)^{\alpha} {\rm d}x \nonumber\\
\qquad{}	\le C N^p \big(|f'(1)|^p+|f'(-1)|^p\big).\label{des-3}
	\end{gather}

For \eqref{des-2} we suppose $M>0$, because in other case this element does not appear in the norm. From Lemmas~\ref{acotacionQ} and \ref{extremos}, for $x\in [-1,1]$, we have
\begin{gather*}
	|L_n(x,1)|\le C\big(1-x^2\big)^{-\frac{\alpha}{2}-\frac{1}{4}}, \qquad |L_n(x,-1)|\le C\big(1-x^2\big)^{-\frac{\alpha}{2}-\frac{1}{4}}.
\end{gather*}
Then \eqref{des-2} is deduced immediately because the integral
\begin{gather}\label{ec:int}
\int_{-1}^{1}\big(1-x^2\big)^{-\frac{p}{2}(\alpha+1/2)+\alpha} {\rm d}x
\end{gather}
is finite for $p<4(\alpha+1)/(2\alpha+1)$.

To prove \eqref{des-3} we suppose $N>0$, because if $N=0$ the inequality is trivially true.
Again, by Lemmas \ref{acotacionQ} and \ref{extremos}, for $x\in [-1,1]$, we obtain the bounds
\begin{gather*}
\left|\frac{\partial L_n}{\partial y}(x,1)\right|\le C\big(1-x^2\big)^{-\frac{\alpha}{2}-\frac{1}{4}},\qquad
\left|\frac{\partial L_n}{\partial y}(x,-1)\right|\le C\big(1-x^2\big)^{-\frac{\alpha}{2}-\frac{1}{4}}.
\end{gather*}
Then, as in the previous case, \eqref{des-3} is a consequence of the finiteness of the integral \eqref{ec:int}.
\end{proof}

\begin{proof}[Proof of Proposition \ref{prop2}]
We are going to show the estimates
\begin{gather}\label{ec:prop2-1}
|G_nf(1)|\le C \|f\|_{W_p^\alpha}, \qquad \text{for} \quad M>0,
\end{gather}
and
\begin{gather}\label{ec:prop2-2}
|(G_nf)'(1)|\le C \|f\|_{W_p^\alpha}, \qquad \text{for} \quad N>0.
\end{gather}
The analysis of $|G_nf(-1)|$, for $M>0$, and $|(G_nf)'(-1)|$, for $N>0$, are completely similar and the details will be omitted.

It is clear that	
\begin{gather*}
	G_nf(1)=\int_{-1}^{1}f(y)L_n(1,y) {\rm d}\mu_\alpha(y)
	+M(f(1)L_n(1,1)+f(-1)L_n(1,-1))\\
\hphantom{G_nf(1)=}{} +N\left(f'(1)\frac{\partial L_n}{\partial y}(1,1)+f'(-1)\frac{\partial L_n}{\partial y}(1,-1)\right).
\end{gather*}
If $M>0$, from Lemmas \ref{acotacionQ} and \ref{extremos} it is obtained that
\begin{gather*}
	|L_n(1,y)|\le C\big(1-y^2\big)^{-\frac{\alpha}{2}-\frac{1}{4}}, \qquad y\in [-1,1].
\end{gather*}
Then, applying H\"{o}lder inequality, we have
\begin{gather*}
	\left|\int_{-1}^{1}f(y)L_n(1,y) {\rm d}\mu_\alpha(y)\right|
	 \le C\|f\|_{L^p({\rm d}\mu_{\alpha})}\left(\int_{-1}^1\big(1-y^2\big)^{-\frac{q}{2}(\alpha+1/2)+\alpha}{\rm d}y\right)^{p/q}
 \le C\|f\|_{L^p({\rm d}\mu_{\alpha})}
\end{gather*}
because	the last integral converges if $q<4(\alpha+1)/(2\alpha+1)$, which is equivalent to $p>4(\alpha+1)/(2\alpha+3)$. On the other hand, using again Lemma~\ref{extremos} we deduce the bounds
\begin{gather*}
|L_n(1,1)|\le C,\qquad |L_n(1,-1)|\le C, \qquad \text{for} \quad N\ge 0
\end{gather*}
and
\begin{gather*}
\left|\frac{\partial L_n}{\partial y}(1,1)\right|\le C,\qquad \left|\frac{\partial L_n}{\partial y}(1,-1)\right|\le C, \qquad \text{for} \quad N>0,
\end{gather*}
which imply, analyzing separately the cases $N>0$ and $N=0$,
\begin{gather*}
\left |M(f(1)L_n(1,1)+f(-1)L_n(1,-1))+N\left(f'(1)\frac{\partial L_n}{\partial y}(1,1)+f'(-1)\frac{\partial L_n}{\partial y}(1,-1)\right)\right| \\
\qquad{} \le C\big( M(|f(1)|+|f(-1)|)+N(|f'(1)|+|f'(-1)|)\big),
\end{gather*}
and \eqref{ec:prop2-1} is proved.

From the identity
\begin{gather*}
	(G_nf)'(1)=\int_{-1}^{1}f(y)\frac{\partial L_n}{\partial x}(1,y) {\rm d}\mu_\alpha(y)
	+M\left(f(1)\frac{\partial L_n}{\partial x}(1,1)+f(-1)\frac{\partial L_n}{\partial x}(1,-1)\right)\\
\hphantom{(G_nf)'(1)=}{} +N\left(f'(1)\frac{\partial^2 L_n}{\partial x\partial y}(1,1)+f'(-1)\frac{\partial^2 L_n}{\partial x\partial y}(1,-1)\right),
\end{gather*}
and the estimates for $N>0$, deduced from Lemmas \ref{acotacionQ} and \ref{extremos},
\begin{gather*}
\left|\frac{\partial L_n}{\partial x}(1,y)\right|\le C\big(1-y^2\big)^{-\frac{\alpha}{2}-\frac{1}{4}}, \qquad y\in [-1,1],\\
\left|\frac{\partial L_n}{\partial x}(1,1)\right|\le C,\qquad \left|\frac{\partial L_n}{\partial x}(1,-1)\right|\le C,
\end{gather*}
and
\begin{gather*}
\left|\frac{\partial^2 L_n}{\partial x\partial y}(1,1)\right|\le C,\qquad \left|\frac{\partial^2 L_n}{\partial x\partial y}(1,-1)\right|\le C,
\end{gather*}
the proof of \eqref{ec:prop2-2} is obtained in the same way as \eqref{ec:prop2-1}.
\end{proof}

\section{Proof of Theorem \ref{density}}\label{section4}
\begin{proof}
	Let $f\in W_p^{\alpha}$ and $\varepsilon>0$. From \cite[Theorem~4.1]{rodriguez}, we have that the space $C_c^{\infty}([-1,1])$ is dense in $L^p({\rm d}\mu_{\alpha})$. Then, there exists a function $g\in C_c^{\infty}([-1,1])$ such that
\begin{gather*}
	\|f-g\|_{L_p({\rm d}\mu_{\alpha})}<\frac{\tilde{\varepsilon}}{4},
\end{gather*}
with $\tilde{\varepsilon}=\varepsilon/(1+M+N)$. We take now a function $h\in C_c^{\infty}([-1,1])$ that satisfies
\begin{gather*}
	\|h\|_{L_p({\rm d}\mu_{\alpha})}<\frac{\tilde{\varepsilon}}{4},
\end{gather*}
and
\begin{gather*}
	h(1)=f(1)-g(1),\qquad h(-1)=f(-1)-g(-1),
\\
	h'(1)=f'(1)-g'(1),\qquad h'(-1)=f'(-1)-g'(-1).
\end{gather*}
Then
\begin{gather*}
\|f-(g+h)\|^p_{W_p^{\alpha}} =\|f-(g+h)\|^p_{L^p({\rm d}\mu_{\alpha})}\\
\hphantom{\|f-(g+h)\|^p_{W_p^{\alpha}} =}{}
 +M\big(|f(1)-(g+h)(1)|^p+|f(-1)-(g+h)(-1)|^p\big)\\
\hphantom{\|f-(g+h)\|^p_{W_p^{\alpha}} =}{} +N\big(|f'(1)-(g+h)'(1)|^p+|f'(-1)-(g+h)'(-1)|^p\big)
\end{gather*}
and
\begin{gather*}
\|f-(g+h)\|_{W_p^{\alpha}}\le \|f-g\|_{L^p({\rm d}\mu_{\alpha})}+\|h\|_{L^p({\rm d}\mu_{\alpha})}< \frac{\tilde{\varepsilon}}{2}.
\end{gather*}
On the other hand, given $\tilde{\varepsilon}$ there exists a polynomial $q_n$ of degree $n$ such that
\begin{gather*}
	\|g+h-q_n\|_{\infty}<\frac{\tilde{\varepsilon}}{2},\qquad \|(g+h)'-q_n'\|_{\infty}<\frac{\tilde{\varepsilon}}{2}.
\end{gather*}
Then, $\|f-q_n\|_{W_p}^p <\varepsilon$ and the proof is completed.
\end{proof}

\subsection*{Acknowledgements} The authors are highly indebted to professor J.M.~Rodr\'{\i}guez for his helpful comments about the proof of Theorem~\ref{density}. The authors were supported by grant MTM2015-65888-C04-4-P from Spanish Government.

\pdfbookmark[1]{References}{ref}
\LastPageEnding


\begin{thebibliography}{99}
\footnotesize\itemsep=0pt

\bibitem{bavinck}
Bavinck H., Meijer H.G., Orthogonal polynomials with respect to a symmetric
 inner product involving derivatives, \href{https://doi.org/10.1080/00036818908839864}{\textit{Appl. Anal.}} \textbf{33} (1989),
 103--117.

\bibitem{Fejzullahu}
Fejzullahu B.X., Marcell\'an F., A {C}ohen type inequality for
 {G}egenbauer--{S}obolev expansions, \href{https://doi.org/10.1216/RMJ-2013-43-1-135}{\textit{Rocky Mountain~J. Math.}}
 \textbf{43} (2013), 135--148.

\bibitem{GPRV1}
Guadalupe J.J., P\'erez M., Ruiz F.J., Varona J.L., Weighted
 {$L^p$}-boundedness of {F}ourier series with respect to generalized {J}acobi
 weights, \href{https://doi.org/10.5565/PUBLMAT_35291_08}{\textit{Publ. Mat.}} \textbf{35} (1991), 449--459.

\bibitem{GPRV2}
Guadalupe J.J., P\'erez M., Ruiz F.J., Varona J.L., Weighted norm inequalities
 for polynomial expansions associated to some measures with mass points,
 \href{https://doi.org/10.1007/s003659900018}{\textit{Constr. Approx.}} \textbf{12} (1996), 341--360,
 \href{https://arxiv.org/abs/math.CA/9505214}{math.CA/9505214}.

\bibitem{MQU}
Marcell\'an F., Quintana Y., Urieles A., On the {P}ollard decomposition method
 applied to some {J}acobi--{S}obolev expansions, \textit{Turkish~J. Math.}
 \textbf{37} (2013), 934--948.

\bibitem{Marce-Xu}
Marcell\'an F., Xu Y., On {S}obolev orthogonal polynomials, \href{https://doi.org/10.1016/j.exmath.2014.10.002}{\textit{Expo.
 Math.}} \textbf{33} (2015), 308--352, \href{https://arxiv.org/abs/1403.6249}{arXiv:1403.6249}.

\bibitem{foulquie}
Moreno A.F., Marcell\'an F., Osilenker B.P., Estimates for polynomials
 orthogonal with respect to some {G}egenbauer--{S}obolev type inner product,
 \href{https://doi.org/10.1155/S1025583499000260}{\textit{J.~Inequal. Appl.}} \textbf{3} (1999), 401--419.

\bibitem{muckenhoupt1}
Muckenhoupt B., Mean convergence of {J}acobi series, \href{https://doi.org/10.2307/2037162}{\textit{Proc. Amer. Math.
 Soc.}} \textbf{23} (1969), 306--310.

\bibitem{muckenhoupt}
Muckenhoupt B., Transplantation theorems and multiplier theorems for {J}acobi
 series, \href{https://doi.org/10.1090/memo/0356}{\textit{Mem. Amer. Math. Soc.}} \textbf{64} (1986), iv+86~pages.

\bibitem{Nevai}
Nevai P., Mean convergence of {L}agrange interpolation.~{III}, \href{https://doi.org/10.2307/1999259}{\textit{Trans.
 Amer. Math. Soc.}} \textbf{282} (1984), 669--698.

\bibitem{PollardII}
Pollard H., The mean convergence of orthogonal series.~{II}, \href{https://doi.org/10.2307/1990435}{\textit{Trans.
 Amer. Math. Soc.}} \textbf{63} (1948), 355--367.

\bibitem{PollardIII}
Pollard H., The mean convergence of orthogonal series.~{III}, \href{https://doi.org/10.1215/S0012-7094-49-01619-1}{\textit{Duke
 Math.~J.}} \textbf{16} (1949), 189--191.

\bibitem{rodriguez}
Rodr\'{\i}guez J.M., Romera E., Pestana D., Alvarez V., Generalized weighted
 {S}obolev spaces and applications to {S}obolev orthogonal polynomials.~{II},
 \href{https://doi.org/10.1007/BF02837397}{\textit{Approx. Theory Appl.}} \textbf{18} (2002), 1--32.

\bibitem{Szego}
Szeg\H{o} G., Orthogonal polynomials, \textit{American Mathematical Society,
 Colloquium Publications}, Vol.~23, 4th~ed., Amer. Math. Soc., Providence, R.I., 1975.

\bibitem{Xu-93}
Xu Y., Mean convergence of generalized {J}acobi series and interpolating
 polynomials.~{I}, \href{https://doi.org/10.1006/jath.1993.1019}{\textit{J.~Approx. Theory}} \textbf{72} (1993), 237--251.

\bibitem{Xu-94}
Xu Y., Mean convergence of generalized {J}acobi series and interpolating
 polynomials.~{II}, \href{https://doi.org/10.1006/jath.1994.1006}{\textit{J.~Approx. Theory}} \textbf{76} (1994), 77--92.

\end{thebibliography}
\end{document}